\documentclass[psamsfonts]{amsart} 
\usepackage{amssymb,latexsym,cite,enumerate,} 
\usepackage[mathscr]{eucal} 
\usepackage[cp1251]{inputenc} 
\usepackage[english]{babel} 
\usepackage{graphicx,xcolor}

\newtheorem{theorem}{Theorem} 
 
\newtheorem{lemma}{Lemma} 
\theoremstyle{definition}

\theoremstyle{remark}

\renewcommand{\P}{\mathbf{P}\,}

\newcommand{\C}{\mathbb{C}}

\allowdisplaybreaks 
\begin{document} 
\title[Roots of Random Polynomials]{On the Distribution of Complex Roots of Random Polynomials with Heavy-tailed Coefficients} 
\author{F. G\"otze, D. Zaporozhets} 
\begin{abstract} 
 
Consider a random polynomial $G_n(z)=\xi_nz^n+\dots+\xi_1z+\xi_0$ 
 with i.i.d. complex-valued coefficients. 
Suppose that the distribution of $\log(1+\log(1+|\xi_0|))$ has a slowly varying tail. Then the distribution of the complex roots of $G_n$ concentrates in probability, as $n\to\infty$, to two centered circles and is uniform in the argument as $n\to\infty$. The radii of the circles are $|\xi_0/\xi_\tau|^{1/\tau}$ and $|\xi_\tau/\xi_n|^{1/(n-\tau)}$, where $\xi_\tau$ denotes the coefficient with the maximum modulus.

\emph{Key words and concepts:} roots of a random polynomial, roots concentration, heavy-tailed coefficients

\end{abstract} 
 
 
\thanks{Partially supported by RFBR (10-01-00242), RFBR-DFG (09-0191331), NSh-4472.2010.1, and CRC 701 ``Spectral Structures and Topological Methods in Mathematics''} 
 \maketitle

\section{Introduction} 
Consider the sequence of random polynomials 
$$ 
G_n(z)=\xi_n z^n+\xi_{n-1} z^{n-1}+\dots +\xi_1z+\xi_0, 
$$ 
where $\xi_0,\xi_1,\dots,\xi_n,\dots$ are i.i.d. real- or complex-valued random variables. We would like to investigate the behaviour of the complex roots of $G_n$. 
 
The first results in this questions  are due to Hammersley  \cite{jH56}. He derived an explicit formula for the $r-$point correlation function ($1\leqslant r\leqslant n$) of the roots of $G_n$  when the coefficients have an arbitrary joint distribution. 

Shparo and Shur \cite{SS62} showed that under quite general assumptions the roots of $G_n$ concentrate near the unit circle as $n$ tends to $\infty$ with asymptotically uniform distribution of the argument. More precisely, denote by $R_n (a,b)$ respectively  $S_n(\alpha,\beta)$ the  number of the roots of $G_n$ contained in the ring $\{z\in\mathbb{C}:a\leq |z|\leq b\}$ respectively   the sector $\{z\in\mathbb{C}:\alpha\leq\arg z\leq\beta\}$. For  $\varepsilon >0, m\in\mathbb{Z}_+$ consider the function  
$$ 
f(t)=\left[\underbrace{\log^+\log^+\dots\log^+t}_{m+1}\right]^{1+\varepsilon}\cdot\prod_{k=1}^m\underbrace{\log^+\log^+\dots\log^+t}_{k}, 
$$ 
where $\log^+s=\max(1,\log s)$.  If for some $\varepsilon>0, m\in\mathbb{Z^+}$ 
$$ 
\mathbf{E}f(|\xi_0|)<\infty , 
$$ 
 then for any $\delta\in(0,1)$ and any $\alpha,\beta$ such that $-\pi\leqslant\alpha<\beta\leqslant\pi$ 
$$ 
\frac1n 
R_n(1-\delta,1+\delta)\stackrel{\mathbf{P}}{\longrightarrow}1,\quad n\to\infty, 
$$ 
$$ 
 \frac1n S_n(\alpha,\beta)\stackrel{\mathbf{P}}{\longrightarrow}\frac{\beta-\alpha}{2\pi},\quad n\to\infty. 
$$ 
 
Ibragimov and Zaporozhets \cite{IZ11} improved this result as follows. They showed that 
$$ 
\P\left\{\frac1nR_n(1-\delta,1+\delta)\underset{n\to\infty}{\longrightarrow}1\right\}=1 
$$ 
holds for any $\delta\in(0,1)$ if and only if 
$$ 
\mathbf{E}\log(1+|\xi_0|)<\infty. 
$$ 
They also proved that for any $\alpha,\beta$ such that  $-\pi\leqslant\alpha<\beta\leqslant\pi$ 
$$ 
\P\left\{\frac1nS_n(\alpha,\beta)\underset{n\to\infty}{\longrightarrow}\frac{\beta-\alpha}{2\pi}\right\}=1 
$$ 
holds for {\em any} distribution of $\xi_0$. 
 
Shepp and Vanderbei \cite{SV95} considered real-valued standard Gaussian coefficients and proved that  
$$ 
\frac1n \mathbf{E}R_n(e^{-\delta/n},e^{\delta/n})\longrightarrow \frac{1+e^{-2 \delta}}{1-e^{-2 \delta}}-\frac 1{\delta},\quad n\to\infty 
$$ 
for any $\delta>0$. Ibragimov and Zeitouni \cite{IZ95} extended this relation to the case of arbitrary i.i.d. coefficients from the domain of attraction of an $\alpha$-stable law: 
\begin{equation}\label{1933} 
\frac1n \mathbf{E}R_n(e^{-\delta/n},e^{\delta/n})\longrightarrow \frac{1+e^{-\alpha \delta}}{1-e^{-\alpha\delta}}-\frac2{\alpha\delta},\quad n\to\infty. 
\end{equation} 
It is interesting to consider the limit case when  $\alpha\to0$. Then  
$$ 
\frac{1+e^{-\alpha \delta}}{1-e^{-\alpha\delta}}-\frac2{\alpha\delta}\longrightarrow 0 
$$ 
and a natural assumption for the coefficient distribution would be a slowly varying tail. In this case \eqref{1933} becomes 
$$ 
\frac1n \mathbf{E}R_n(e^{-\delta/n},e^{\delta/n})\longrightarrow 0,\quad n\to\infty. 
$$ 
This result (in a slightly stronger form) is proved in Theorem \ref{1934}.

In contrast to the concentration near the unit circumference, there exist random polynomials with quite a different asymptotic behavior of complex roots. Zaporozhets \cite{dZ05} constructed a random polynomial with i.i.d. coefficients such that in average $n/2+o(1)$ of the complex roots concentrate near the origin and the same number tends to infinity as $n\to\infty$ (moreover, the expected number of {\it real}  roots of this polynomial is at most 9 for all $n$). Theorem~\ref{2147} generalizes this result. 
 
The paper is organized as follows. In Sect.~\ref{2150} we formulate our results. In Sect.~\ref{2152} we prove some auxiliary lemmas. The theorems are proved in Sect.~\ref{2208}.

By $\sum_{j}$ we always denote a summation taken over all $j$ from $\{0,1,\dots,n\}$. If conditions are stated for the summation, they are applied to this default range $j$  from  $\{0,1,\dots,n\}$.
 
\section{Results}\label{2150} 
 
For the sake of simplicity, we assume that $\P\{\xi_0=0\}=0$.  To treat the general case it is enough to study in the same way the behavior of the roots on the sets $\{\alpha_n =k,\beta_n=l\}$, where 
$$ 
\alpha_n=\max\{j=0,\dots,n\,:\,\xi_{j}\ne0\},\qquad 
\beta_n=\min\{j=0,\dots,n\,:\,\xi_{j}\ne0\}. 
$$ 
\begin{theorem}\label{1934} 
If the distribution of $|\xi_0|$ has a slowly varying tail, then for any  $\delta>0$ 
$$ 
\P\{R_n(e^{-\delta/n},e^{\delta/n})=0\}\longrightarrow 1,\quad n\to\infty. 
$$ 
\end{theorem} 
 
Consider the index $\tau=\tau_n\in\{0,\dots,n\}$ such that $|\xi_\tau|\geqslant|\xi_j|$ for $j=0,\dots,n$. If it is not unique, we take the minimum one. Let $\omega_1,\dots,\omega_n$ be the complex roots of the system of equations 
$$ 
z^{\tau}+\frac{\xi_{0}}{\xi_{\tau}}=0,\quad z^{n-\tau}+\frac{\xi_{\tau}}{\xi_{n}}=0. 
$$ 
 
\begin{theorem}\label{2147} 
If the distribution of $\log(1+\log(1+|\xi_0|))$ has a slowly varying tail, then for any  $\varepsilon\in(0,1)$ 
$$ 
\mathbf{P}\{F_n(\varepsilon)\}\to1,\quad n\to\infty, 
$$ 
where  $F_n(\varepsilon)$ denotes the event that it is possible to enumerate the roots $z_1,\dots,z_n$ of $G_n$ in such a way that  
$$ 
|z_k-w_k|<\frac\varepsilon n |w_k| 
$$ 
for $k=1,\dots,n$. 
\end{theorem}

\section{Auxiliary lemmas}\label{2152} 
First we need to formulate and prove some auxiliary results. The following result is due to Pellet. 
\begin{lemma}\label{1354} 
Let $g(z)=\sum_{j}a_jz^j$ be a polynomial of degree $n$.  Suppose for some $k=1,\dots,n-1$ the associated polynomial 
$$ 
\tilde{g}(z)=\sum_{j\ne k}|a_j|z^j-|a_k|z^k 
$$ 
has exactly two positive roots $R$ and $r$, $R>r$. Then $g$ has exactly $k$ roots inside the circle $\{z\in\C\,:\,|z|=r\}$ and $n-k$ roots 
 outside the circle $\{z\in\C\,:\,|z|=R\}$. 
\end{lemma} 
\begin{proof} 
See, e.g., \cite{RS02}.  
\end{proof} 
 
The next lemma is due to Ostrowski. 
\begin{lemma}\label{1900} 
Let $B$ be a closed region in the complex plane,
 the boundary of which consists of a finite number of regular arcs; 
let the functions $f(z), h(z)$ be regular on $B$.
 Assume that for all values of the real parameter $t$, running in the interval $a\leqslant t\leqslant b$, the function $f(z)+t\cdot h(z)$ is non zero on the boundary of $B$. Then the number of the roots of $f(z)+t\cdot h(z)$  inside $B$ is independent of $t$ for $a\leqslant t\leqslant b$. 
\end{lemma} 
\begin{proof} 
See \cite{aO73}. 
\end{proof} 
 
\begin{lemma}\label{1935} 
Consider a monic  polynomial of degree $n$ with complex coefficient  $g(z)=\sum_{j}a_jz^j$ such that $a_n=1,a_0\ne0$. Fix some $k=1,\dots,n-1$ and denote by $w_1,\dots,w_{n-k}$ the roots of the equation $z^{n-k}+a_k=0$. Put  
$$ 
A_k=\sum_{j\ne k}|a_j|. 
$$  
If for some $\varepsilon>0$ 
\begin{equation}\label{1817} 
A_k\leqslant\left(1-\frac\varepsilon n\right)\left(\frac\varepsilon{n+\varepsilon}\right)^{n-k}|a_k|^{1/(n-k)}, 
\end{equation} 
then $g$ has exactly $n-k$ roots $z_1,\dots,z_{n-k}$ outside the unit circumference and it is possible to enumerate these roots in such a way that  
$$ 
|z_j-w_j|\leqslant\frac\varepsilon n|w_j| 
$$ 
for $j=1,\dots,n-k$. 
\end{lemma} 
 
\begin{proof} 
We will prove a stronger version of the Lemma~\ref{1935}. Namely, we will show that the statement holds for the family of polynomials 
$$ 
g_t(z)=z^n+a_kz^k+t\sum_{j\ne k,n}a_jz^j,\quad 0\leqslant t\leqslant1. 
$$ 
In particular, 
$$ 
g_0(z)=z^n+a_kz^k,\quad g_1(z)=g(z). 
$$ 
 
Let us use Lemma \ref{1354} to estimate absolute values of the roots of $g_t$. Consider the associated polynomial 
$$ 
\tilde{g_t}(z)=z^n-|a_k|z^k+t\sum_{j\ne k,n}|a_j|z^j. 
$$ 
We have $\tilde{g_t}(0),\tilde{g_t}(\infty)>0$ and it follows from \eqref{1817} that $\tilde{g_t}(1)<1$. Also, by Descarte's rule of signs, $\tilde{g_t}$  has at most 2 positive roots. Therefore $\tilde{g_t}$  has exactly 2 positive roots  $r_t$ and $R_t$  such that 
\begin{equation}\label{9} 
0<r_t<1<R_t. 
\end{equation} 
 
Now let us show that 
\begin{equation}\label{10} 
\left(1-\frac\varepsilon n\right)|a_k|^{1/(n-k)}\leqslant R_t\leqslant |a_k|^{1/(n-k)}. 
\end{equation} 
Since $\tilde{g_t}(R_t)=0$, we have 
\begin{equation}\label{11} 
R_t^{n-k}+t\sum_{j\ne k,n}|a_j|R_t^{j-k}=|a_k|, 
\end{equation} 
which proves the right side of  \eqref{10}.  
 
We prove the left side by contradiction. Suppose, on the contrary,  that 
$$ 
R_t<\left(1-\frac\varepsilon n\right)|a_k|^{1/(n-k)}. 
$$ 
Then 
\begin{multline*} 
R_t^{n-k}+t\sum_{j\ne k,n}|a_j|R_t^{j-k}<\left(1-\frac\varepsilon n\right)^{n-k}|a_k|+A_kR_t^{n-k-1} 
\\\leqslant\left(1-\frac\varepsilon n\right)^{n-k}|a_k|+A_k\left(1-\frac\varepsilon n\right)^{n-k-1}|a_k|^{1-\frac{1}{n-k}} 
\\=\left(1-\frac\varepsilon n\right)^{n-k}|a_k|+\frac{A_k}{|a_k|^{1/(n-k)}}\left(1-\frac\varepsilon n\right)^{n-k-1}|a_k|. 
\end{multline*} 
It follows from \eqref{1817} that  
$$ 
\frac{A_k}{|a_k|^{1/(n-k)}}\leqslant\frac\varepsilon n, 
$$ 
therefore, 
\begin{multline*} 
R_t^{n-k}+t\sum_{j\ne k,n}|a_j|R_t^{j-k}<\left(1-\frac\varepsilon n\right)^{n-k}|a_k|+\frac{\varepsilon}{n}\left(1-\frac\varepsilon n\right)^{n-k-1}|a_k|\\ 
=\left(1-\frac\varepsilon n\right)^{n-k-1}|a_k|\leqslant|a_k|, 
\end{multline*} 
which contradicts with \eqref{11}. Thus \eqref{10} is proved. 
 
It follows from  \eqref{9}, \eqref{10} and the Lemma~\ref{1354} that  $k$ roots of $g_t$ lie inside the circle $\{z\in\C\,:\,|z|=1\}$ and the other $n-k$ -- outside the circle $\{z\in\C\,:\,|z|=(1-\varepsilon/n)|a_k|^{1/(n-k)}\}$ for all $t\in[0,1]$.  
 
Let $z_0$ be a root of $g_t$ from the second group, i.e., 
\begin{equation}\label{1831} 
|z_0|>\left(1-\frac\varepsilon n\right)|a_k|^{1/(n-k)}. 
\end{equation} 
We have 
\[ 
|z_0^n+a_kz_0^k|=t\cdot\Big|\sum_{j\ne k,n}a_jz_0^j\Big|\leqslant A_kz_0^{n-1}, 
\] 
which leads to 
\[ 
\prod_{j=1}^{n-k}|z_0-w_j|\leqslant A_k|z_0|^{n-k-1}. 
\] 
This implies that there exists an index $l$ such that 
$$ 
|z_0-w_l|\leqslant \left(\frac{A_k}{|z_0|}\right)^{1/(n-k)}|z_0|. 
$$ 
Combining this with \eqref{1817} and \eqref{1831}  we obtain 
\begin{multline*} 
|z_0-w_l|<\left(\frac{A_k}{(1-\varepsilon/n)|a_k|^{1/(n-k)}}\right)^{1/(n-k)}|z_0|\\\leqslant\frac\varepsilon{n+\varepsilon}|z_0| 
\leqslant\frac\varepsilon{n+\varepsilon}|w_l|+\frac\varepsilon{n+\varepsilon}|z_0-w_l|, 
\end{multline*} 
which produces 
$$ 
|z_0-w_l|<\frac\varepsilon n|w_l|=\frac\varepsilon n|a_k|^{1/(n-k)}. 
$$

It means that all roots of $g_t$ from the second group belong to $\cup_{m=1}^{n-k}B_m$, where $B_m=\{z\in\C\,:\,|z-w_m|<\varepsilon|w_m|/n\}$. Since  $\varepsilon/n<\sin[\pi/(n-k)]$,  all $B_1,\dots,B_{n-k}$ are disjoint. Therefore $g_t$ does not vanish on the boundary of $B_m$ for all $t\in[0,1],m=1,\dots,n-k$. To conclude the proof, it remains to show that every $B_m$ contains  exactly one root of $g_t$. Obviously, this is true for $t=0$. Therefore, by Lemma \ref{1900}, this is also true for all $t\in[0,1]$. 
\end{proof} 
 
\begin{lemma}\label{0728} 
Let $\{\eta_j\}_{j=0}^\infty$ be non-negative i.i.d. random variables. Put 
$$ 
S_n=\sum_{j}\eta_j,\quad M_n=\max\{\eta_j\}_{j=0}^n. 
$$  
\begin{itemize} 
\item[(a)] \label{2034} The distribution of $\eta_0$ has a slowly varying tail if and only if 
$$ 
\frac{M_n}{S_n}\overset{\text{P}}{\to}1,\quad n\to\infty. 
$$ 
\item[(b)] \label{2035} The distributin of $\eta_0$ has an infinite mean if an only if 
$$ 
\frac{S_n-M_n}{n}\overset{\text{a.s.}}{\to}\infty,\quad n\to\infty. 
$$ 
\end{itemize} 
\end{lemma} 
\begin{proof} 
For (a) see \cite{dD52}, for (b) see \cite[Theorem 2.1]{kM1995}. 
\end{proof} 
 
\begin{lemma}\label{1918} 
Suppose $a_0,a_1,\dots,a_n\geqslant0$ and $\varepsilon>0$. If for some $k=1,\dots,n-1$ 
$$ 
\prod_{j\ne k}(1+a_j)^{2n^2}\leqslant1+a_k 
$$ 
and  
\begin{equation}\label{1554} 
a_k\geqslant2(1-\varepsilon)^{-4n^2/(4n-1)}\varepsilon^{-4n^3/(4n-1)}(n+\varepsilon)^{4n^3/(4n-1)}, 
\end{equation} 
then 
$$ 
\sum_{j\ne k}a_j+1\leqslant\left(1-\frac\varepsilon n\right)\left(\frac\varepsilon{n+\varepsilon}\right)^{n-k}a_k^{1/(n-k)}. 
$$ 
\end{lemma} 
\begin{proof} 
Since $1+\sum_{j\ne k}a_j\leqslant\prod_{j\ne k}(1+a_j)$, 	it suffices to show that 
$$ 
(2a_k)^{1/(2n)^2}\leqslant(1-\varepsilon)\left(\frac\varepsilon{n+\varepsilon}\right)^{n} a_k^{1/n}, 
$$ 
which is equivalent to \eqref{1554}. 
\end{proof}

\section{Proof of theorems}\label{2208} 
\begin{proof}[Proof of Theorem~\ref{1934}] 
By Lemma \ref{0728} (a), for any $\delta>0$ we have $\P\{A_n\}\to1,n\to\infty$, where 
$$ 
A_n=\left\{|\xi_\tau|>e^{\delta}\sum_{j\ne\tau}|\xi_j|\right\}. 
$$ 
 
Consider the associated polynomial 
$$ 
\tilde{G}(z)=\sum_{j\ne\tau}|\xi_j|z^j-|\xi_\tau|z^\tau. 
$$ 
 
Suppose $A_n$ occurs. If $1\leqslant t\leqslant e^{\delta/n}$, then 
$$ 
|\xi_\tau t^\tau|>e^{\delta}\sum_{j\ne\tau}|\xi_j|\geqslant t^n\sum_{j\ne\tau}|\xi_j|\geqslant\Big|\sum_{j\ne\tau}t^j\xi_j\Big|. 
$$ 
If $e^{-\delta/n}\leqslant t\leqslant1$, then 
$$ 
|\xi_\tau t^\tau|\geqslant e^{-\delta}|\xi_\tau|>\sum_{j\ne\tau}|\xi_j|\geqslant\Big|\sum_{j\ne\tau}t^j\xi_j\Big|. 
$$ 
Therefore  $\tilde{G}$ does not have real roots in the interval $[e^{-\delta/n},e^{\delta/n}]$. Further, $\tilde{G}(0)>0,\tilde{G}(\infty)>0$, and $\tilde{G}(1)<0$. By Descarte's rule of signs $\tilde{G}$  has at most 2 positive roots. Thus $\tilde{G}$  has exactly 2 positive roots  $r$ and $R$  such that 
$$ 
0<r<e^{-\delta/n}<e^{\delta/n}<R. 
$$ 
 
By Lemma \ref{1354},  $G$ has exactly $\tau$ roots inside the circle $\{z\in\C : |z|=e^{-\delta/n}\}$ and $n-\tau$ roots 
 outside the circle $\{z\in\C : |z|=e^{\delta/n}\}$. Therefore, $A_n$ implies that $R_n(e^{-\delta/n},e^{\delta/n})=0$ which concludes the proof. 
\end{proof}

\begin{proof}[Proof of Theorem~\ref{2147}] 
Consider the events 
$$ 
A_n=\left\{\prod_{j\ne\tau}\left(1+\frac{|\xi_j|}{|\xi_n|}\right)^{2n^2}\leqslant1+\frac{|\xi_\tau|}{|\xi_n|}\right\} 
$$ 
and 
$$ 
B_n=\left\{\frac{|\xi_\tau|}{|\xi_n|}\geqslant2(1-\varepsilon)^{-4n^2/(4n-1)}\varepsilon^{-4n^3/(4n-1)}(n+\varepsilon)^{4n^3/(4n-1)}\right\}. 
$$ 
Since the distribution of $\log(1+\log(1+|\xi_0|))$ has a slowly varying tail, by Lemma \ref{0728} (a), 
$$ 
\mathbf{P}\Bigg\{4\cdot\sum_{j\ne\tau}\log(1+\log(1+|\xi_j|))\leqslant\log(1+\log(1+|\xi_{\tau}|))\Bigg\}\to 1,\quad n\to\infty, 
$$ 
which implies 
\begin{equation}\label{35} 
\mathbf{P}\Bigg\{\bigg(\sum_{j\ne\tau}\log(1+|\xi_j|)\bigg)^4\leqslant\log(1+|\xi_{\tau}|)\Bigg\}\to 1,\quad n\to\infty. 
\end{equation} 
Since $\mathbf{E}\log(1+|\xi_0|)=\infty$, by Lemma \ref{0728} (b) with probability one  
$$ 
\frac1n\sum_{j\ne\tau}\log(1+|\xi_j|)\to\infty,\quad n\to\infty, 
$$ 
which together with \eqref{35} produces 
$$ 
\mathbf{P}\left\{n^3\cdot\sum_{j\ne\tau}\log(1+|\xi_j|)\leqslant\log(1+|\xi_{\tau}|)\right\}\to 1,\quad n\to\infty, 
$$ 
and 
$$ 
\mathbf{P}\left\{\log(1+|\xi_{\tau}|)\geqslant n^4\right\}\to 1,\quad n\to\infty. 
$$ 
Since for any $\delta>0$ there exists $T>0$ such that $\P\{T^{-1}<|\xi_n|<T\}>1-\delta$, the last two inequalities imply 
$$ 
\P\{A_n\},\P\{B_n\}\to1,\quad n\to\infty. 
$$ 
 
By Lemma~\ref{1918}, the event $A_n\cap B_n$ implies that the polynomial $G_n(z)/|\xi_n|$ satisfies the conditions of Lemma~\ref{1935}.  Thus we have proved the theorem for the roots of $G_n$ lying outside the unit circumference. To treat  the rest of the roots consider the associated polynomial  
$$ 
G_n^*(z)=z^nG(1/z)=\sum_{j}\xi_{j}z^{n-j} 
$$ 
and note that $z_0$ is a root of $G_n$ if and only if $z^{-1}$ is a root of $G_n^*(z)$.

\end{proof}

\end{document}